\documentclass{amsart}

\usepackage{amssymb,amsmath,latexsym,graphics,amsfonts}

\newcommand{\su}{\subseteq}

\newtheorem{theorem}{Theorem}[section]
\newtheorem{corollary}[theorem]{Corollary}
\newtheorem{lemma}[theorem]{Lemma}
\newtheorem{proposition}[theorem]{Proposition}
\theoremstyle{definition}
\newtheorem{definition}[theorem]{Definition}
\theoremstyle{definition}

\newtheorem{remark}[theorem]{Remark}

\newtheorem{pro}[theorem]{Problem}
\numberwithin{equation}{section}

\begin{document}

% Make sure to set *all* of the following

\date{}

\title{Decomposition of Euclidean Nearly K\"ahler submanifolds}                                     %<-------------------
\author[Heidari, Heydari]{Nikrooz Heidari, Abbas Heydari}                 %<-------------------

% Mathematics Subject Classification 2000

    %<-------------------

% Keywords and phrases
\keywords{Nearly K\"ahler manifold, Isometric immersion, Leaf space, Euclidean submanifold, \\
%{\it Acknowledgement}: The authors would like to thank the anonymous referee for
%useful comments and corrections.
}         %<-------------------

% For each author, add one of the following

\address{%
{\bf Nikrooz Heidari}\\ Department of Mathematics\\
Faculty of Mathematical Sciences
\\Tarbiat Modares University\\Tehran 14115-134, Iran\\
nikrooz.heidari@modars.ac.ir}

\address{%
{\bf Abbas Heydari}\\ Department of Mathematics\\
Faculty of Mathematical Sciences
\\ Tarbiat Modares University \\Tehran 14115-134, Iran\\
aheydari@modares.ac.ir}

\maketitle
\begin{abstract}
 We study the foliation space of complex and invariant (by torsion of intrinsic  Hermitian connection) umbilic distribution on an isometric immersion from a nearly K\"ahler manifold $M$ into the Euclidean space. Under suitable conditions this leaf space is nearly K\"ahler and $M$ can be decomposed into a product of this leaf space and a 6-dimensional locally  homogeneous nearly K\"ahler manifold.
\end{abstract}

\section{Introduction}

 Gray in in his study of weak holonomy groups in 70's introduced nearly K\"ahler manifolds and obtained ceratin relations for the Riemannain curvature operator of such manifolds. These identities are slightly more complicated but resembling the corresponding formulas for the Riemannian curvature operator of K\"ahler manifolds. By using these identities he got many interesting results about the geometry and topology of these manifolds\cite{Gray2, Gray3, Gray4}. He also found a
 large class of homogeneous nearly K\"ahler manifolds using 3-symmetric space, introduced by Gray and Wolf \cite{wol1, wol2}. A 3-symmetric (semi)-Riemannain space is naturally reducible if and only if it is nearly K\"ahler with the canonical almost complex structure \cite{Gray3}. Gray propounded this conjecture ``every nearly K\"ahler homogenous manifold is a 3-symmetric space equipped with its canonical almost complex structure".
 \paragraph*{}
A nearly K\"ahler manifold can be described as an almost Hermitian manifold such that the torsion of intrinsic Hermitian connection is totally skew-symmetric \cite{Cly}. In this point of view, nearly K\"ahler manifolds are interesting objects in string theory \cite{Fri}. In 2002, Nagy proved that every simply connected complete nearly K\"ahler manifold is isometricto a Riemannain
product space $M_{1}\times\cdots\times M_{k}$ where $M_{i}$'s are nearly K\"ahler and belong to the following list:
\begin{enumerate}
\item[(i)] naturally reduce 3-symmetric spaces (these homogeneous spaces are divided into four separated classes \cite{Gon1}),
 \item[(ii)] twistor spaces over positive K\"ahler-quaternion manifolds,
 \item[(iii)] six dimensional nearly K\"ahler manifolds.
 \end{enumerate}
  By this decomposition, Gray conjecture converts to this question: ``Is it true that every 6-dimensional, complete, homogeneous nearly K\"ahler manifold is a 3-symmetric space?" Butruile in 2008 \cite{But} showed that there exist only four complete, homogeneous 6-dimensional nearly K\"ahler manifolds (up to homothety and covering space) and all of them are 3-symmetric:
  \begin{align*}
S^{6}&=\frac{G_2}{SU(3)},\qquad S^{3}\times S^{3}=\frac{SU(2)\times
SU(2)}{<1>},\\
\qquad \mathbb{C}P^{3}&=\frac{Sp(2)}
{SU(2)\cdot U(1)},\qquad
\mathbb{F}^{3}=\frac{SU(3)}{U(1)\times U(1)}
\end{align*}
The only known examples of  6-dimensional nearly K\"ahler manifolds are above four manifolds. An important question in nearly K\"ahler geometry is the fundamental explanation of rareness of such manifolds or difficulties of introducing non-homogeneous examples. This question can be formulated in Butrulle conjecture ``every complete nearly K\"ahler manifold is a 3-symmetric space".\\
 This conjecture can be separated by Nagy's decomposition into two conjectures:
\begin{pro}\label{pro1} Every complete (compact) 6-dimensional nearly K\"ahler manifold is homogenous.
\end{pro}
\begin{pro}\label{pro2} Every positive quaternion-K\"ahler manifold is a Wolf space.
\end{pro}
These statements motivated us to study isometric immersions $f:M^{2n}\longrightarrow\mathbb{Q}^{2n+p}$ from a nearly K\"ahler manifold into a space form (especially the Euclidean space). We introduced in \cite{Hei3} an umbilic distribution which is complex and invariant by the torsion of  intrinsic  Hermitian connection and showed that this distribution is integrable and each leaf of the generated foliation is a 6-dimensional locally homogeneous nearly K\"ahler submanifold. In \cite{Hei4} by description of this foliation, we parameterized the isometric immersion $f$ and produced some examples of nearly K\"ahler submanifolds in the standard space forms with arbitrary co-dimension $p$.\\

In this article, we study the leaf space of the umbilic distribution which is complex and invariant by the torsion of  intrinsic  Hermitian connection. We investigate the existence of decompositions like that of Nagy in the Euclidean submanifolds case. We show that under suitable conditions there is a product decomposition such that the 6-dimensional nearly K\"ahler terms  are (locally) homogeneous.

\section{Preliminaries}
A smooth manifold $M$ is called almost complex if there exists $(1,1)$ tensor field $J$ on $M$ such that $J^{2}=-Id$. A Riemmanian manifold $(M,g)$ with almost complex structure $J$ is called almost Hermitian if $g(JX,JY)=g(X,Y)$ for all vector fields $X$ and $Y$ on $M$.\\
Gray and Hervella \cite{her} classified almost Hermitian manifolds into sixteen classes. One of the most important classes in this classification is the class of K\"ahler manifolds. An almost Hermitian manifold $(M,g,J)$ is a K\"ahler manifold if $\nabla J=0$ where $\nabla$ is the Levi-Civita connection of $g$ and  nearly K\"ahler if $(\nabla_{X}J)X=0$ for all vector filed $X$ on $M$. Every K\"ahler manifold is nearly K\"ahler but the converse is not true. Non-K\"ahler nearly K\"ahler manifolds are called strictly nearly K\"ahler. There is no strictly nearly K\"ahler manifold in dimension less than six.\\
the canonical Hermitian connection on an almost Hermitian manifold defined by
\begin{align*}
\bar{\nabla}_{X}Y=\nabla_{X}Y+\frac{1}{2}(\nabla_{X}J)Y
\end{align*}
Its easy to see that $\bar{\nabla}$ is the unique linear connection on $M$ such that $\bar{\nabla}g=0$ (it is a metric connection) and $\bar{\nabla}J=0$ (it is a Hermitian connection).
\begin{proposition}\cite{But}\label{definition nk}
Let M be an almost Hermitian manifold. The following conditions are equivalent and define a nearly K\"ahler manifold:
\begin{enumerate}
\item[(1)] The torsion $T(X,Y)=(\nabla_{X}J)JY$ of $\bar{\nabla}$ is totally skew-symmetric (equivalently, the tensor $T(X,Y,Z)=g(T(X,Y),Z)$ is skew-symmetric),
\item[(2)] $(\nabla_{X}J)X=0$  for all $X\in TM$,
\item[(3)] $\nabla_{X}\omega=\frac{1}{3}\mathfrak{i}_{X}d\omega$ for all $X\in TM$ where $\omega(X,Y)=g(X,JY)$ is the K\"ahler 2-form on $M$,
\item[(4)] $d\omega$ is of type $(0,3)+(3,0)$ and the Nijenhuis tensor $N$ is totally skew-symmetric.
\end{enumerate}
\end{proposition}
\begin{proposition}\cite{But}
For a nearly K\"ahler manifold, The torsion of the intrinsic Hermitian connection is totally skew-symmetric and parallel, that is $\bar{\nabla}\eta=0$, where $\eta(X)=\frac{1}{2}J\circ(\nabla_{X}J)$. This is equivalent to $\bar{\nabla}\nabla\omega=0$ or $\bar{\nabla}d\omega=0$.
\end{proposition}
‎\begin{lemma}\label{gray formula}‎\cite{Gray2}‎
For a nearly K\"ahler manifold $(M,g,J)$, ‎
‎\begin{eqnarray*}‎
 ‎&(\nabla_{X}J)Y+(\nabla_{Y}J)X=0,\ \ \  ‎
 (\nabla_{JX}J)JY=(\nabla_{X}J)JY, \\‎
 ‎&J(\nabla_{X}J)Y=-(\nabla_{X}J)JY=-(\nabla_{JX}J)Y,‎\ \ \ 
 g(\nabla_{X}Y,X)=g(\nabla_{X}JY‎, ‎JX),\\‎
 ‎&2g((\nabla_{W,X}^2J)Y,Z)=-\sigma_{X,Y,Z}g((\nabla_WJ)X,(\nabla_{Y}J)JZ)‎.
‎\end{eqnarray*}‎
‎\end{lemma}‎
‎Gray used the following relations between the torsion of the intrinsic Hermitian connection and Riemannian curvature of nearly K\"ahler manifolds \cite{Gray4}‎. ‎These formulas resemble the corresponding formulas for K\"ahler manifolds.\\‎
‎\begin{eqnarray*}‎
\langle R_{X,Y}X,Y\rangle-\langle R_{X,Y}JX,JY\rangle=\|(\nabla_{X}J)Y\|^{2}\\‎
‎\langle R_{W,X}Y,Z\rangle-\langle R_{W,X}JY,JZ\rangle=\langle (\nabla_{W}J)X,(\nabla_{Y}J)Z\rangle\\‎
‎\langle R_{W,X}Y,Z\rangle=\langle R_{JW,JX}JY,JZ\rangle\\‎
‎2g((\nabla_{W,X}^{2}J)Y,Z)=\sigma_{X,Y,Z}g(R_{WJX}Y,Z).\\‎
‎\end{eqnarray*}‎
‎The star version of Ricci tensor of metric $g$ is defined by‎
‎\begin{align*}‎
‎\langle Ric^{*}(X),Y\rangle=\frac{1}{2}\sum_{i=1}^{n}R(X‎, ‎JY‎, ‎e_{i}‎, ‎Je_{i}‎),
‎\end{align*}‎
‎where $R$ is the Riemannian curvature of $(M,g)$ and $\{e_{i}\}$ is a local frame field‎. The difference tensor $r$ between $Ric^*$ and Ricci tensor is described by the following formula \cite{Nag2}‎:
‎\begin{align*}‎
‎\langle rX,Y\rangle=\sum_{i=1}^{n}\langle (\nabla_{e_{i}}J)X,(\nabla_{e_{i}}J)Y\rangle‎.
‎\end{align*}‎
It is easy to see that $r$ is symmetric‎, ‎positive and commutes with $J$. The tensor r has strong geometric properties, e.g., Gray in‎ ‏‎\cite{Gray4} proved that‎
‎\begin{align*}‎
‎2\langle (\nabla_{X}r)Y,Z\rangle=\langle r(\nabla_{X}Y,JZ\rangle+\langle r(JY),(\nabla_{X}J)Z\rangle‎
‎\end{align*}‎
in other words, $r$ is $\bar{\nabla}$-parallel (i.e $\bar{\nabla}r=0$).\\‎‎
%\begin{definition}‎\label{definition 1}
A nearly K\"ahler manifold is strictly nearly K\"ahler when the kernel of ‎$‎r‎$ ‎‎vanishes. ‎This is ‎equivalent to triviality of the distribution ‎‎$‎x‎\longmapsto ‎\{X\in ‎T_{x}M|T(X,Y)=0‎, \forall ‎Y\in ‎T_{x}M‎\}$‎‎.
%\end{definition}
‎\begin{proposition}‎‎‎\cite{Nag2}
Let ‎$‎(M,g,J)‎$ ‎be a‎ ‎complete nearly ‎K\"ahler ‎manifold. ‎Then ‎‎$‎M‎$ ‎can ‎be ‎decomposed as a Riemannain product ‎$‎M_{1}\times M_{2}‎$ ‎where ‎‎$‎M_{1}‎$ is a‎ K\"ahler ‎manifold ‎and ‎‎$‎M_{2}‎$ ‎is a‎ ‎strictly nearly ‎K\"ahler ‎manifold.‎
‎\end{proposition}‎‎
On nearly K\"ahler manifold‎, tensors 
\[A(X,Y,Z)=\langle (\nabla_{X}J)Y,Z\rangle,\ \  B(X,Y,Z)=\langle (\nabla_{X}J)Y,JZ\rangle\]
 are skew-symmetric and have type $(3,0)+(0,3)$ as (real) 3-forms.\\‎
‎We need the following classical relation between the covariant derivative of the almost complex structure $J$ and its Nijenhuis tensor $N$ which is proved by a straightforward computation using‎
‎\begin{align*}‎
‎4N(X,Y)=[X,Y]+J[JX,Y]+J[X,JY]-[JX,JY]‎
‎\end{align*}‎
‎and the anti-symmetry of the above tensors A and B‎.
‎\begin{lemma}\label{nijenhus}‎
‎For every nearly K\"ahler manifold $(M,g,J)$ we have‎
‎\begin{align*}‎
N(X,Y)=J(\nabla_{X}J)Y‎.
‎\end{align*}‎
‎\end{lemma}‎
 In lower dimensions, nearly K\"ahler manifolds are mainly determined. If $M$ is a nearly K\"ahler with $dimM\leq4$, then M is K\"ahler. If $dimM=6$, then we have the following result.
\begin{proposition}\cite{Gray2,Gray4,wat1}\label{6-dim}
Let $(M,g, J)$ be a 6-dimensional, strict, nearly K\"ahler manifold. Then
 \begin{enumerate}
\item[(1)] $\nabla J$ has constant type, that is
 \begin{align*}
\|(\nabla_{X}J)Y\|^{2}=\frac{S}{30}(\|X\|^{2}\|Y\|^{2}-g(X,Y)^{2}-g(JX,Y)^{2}),
\end{align*}
\item[(2)] the first Chern class of $(M, J)$ vanishes,
\item[(3)] $M$ is an Einstein manifold with
\begin{align*}
Ricc=\frac{S}{6}g, \qquad Ricc^{*}=\frac{S}{30}g.
\end{align*}
Moreover if the tensor $\nabla J$ has constant type $\alpha$ then $dimM=6$ and $\alpha=\frac{S}{30}$ where $S$ is the scalar curvature.
\end{enumerate}
\end{proposition}
%\begin{remark}‎\label{remark 1}
%noticed that according proposition  ‎‎‎\ref{definition 1} a ‎‎6-dimentional non-k\"ahle and nearly ‎K\"ahler ‎manifolds ‎is a‎ ‎strictly ‎Nearly ‎K\"ahler ‎manifolds.‎‎
%\end{remark}‎
The next lemma follows immediately.
\begin{lemma}
 For vector fields $W, X, Y$ and $Z$ we have
 \begin{eqnarray*}
g((\nabla_{W}J)X,(\nabla_{Y}J)Z)=\frac{S}{30}\{g(W, Y)g(X, Z)-g(W, Z)g(X, Y)\\
-g(W, JY)g(X, JZ)+g(W, JZ)g(X, JY)\},
\end{eqnarray*}
and
\begin{eqnarray*}
 g((\nabla_{W}\nabla_{Z}J)X, Y)=\frac{S}{30}\{g(W, Z)g(JX, Y)
 -g(W, X)g(JZ, Y)+g(W, Y)g(JZ, X)\},
 \end{eqnarray*}
 also
 \begin{eqnarray*}
  \Sigma g(Je_{i}, e_{j})R(e_{i}, e_{j}X, Y)=-\frac{S}{15}g(Jx, Y),\\
\Sigma g((\nabla_{X}J)e_{i}, e_{j})R(e_{i}, e_{j}, Y, Z)=-\frac{S}{30}g((\nabla_{X}J)Y, Z),
  \end{eqnarray*}
where $\{e_{i}\}$ is a local orthonormal frame field on $M$.
\end{lemma}
\begin{lemma}‎\cite{Mor2}‎‎‎\label{adapted frame}‎‎
Let ‎$‎X, Y‎$ ‎be two ‎vector ‎fields ‎on ‎‎$‎M‎$ ‎then ‎‎$‎T(X, Y‎)‎$ ‎is an ‎orthogonal vector field to ‎$‎X, JX, Y‎$ ‎and ‎‎$‎JY‎$.‎‎
\end{lemma}‎
This ‎lemma ‎leads ‎us ‎to ‎define a‎ ‎suitable ‎local ‎frame on a 6-dimensional nearly K\"ahler manifold ‎which is particularly convenient for local calculations. Suppose ‎that ‎‎$‎e_{1},e_{2}‎$ ‎be ‎two ‎orthogonal local vector fields on ‎$‎M‎$ and define‎
‎\begin{align*}‎‎
e_{3}=T(e_{1},e_{3})=(\nabla_{e_{1}}J)e_{2}‎\qquad ‎e_{4}=Je_{1}‎\qquad‎‎ e_{5}=Je_{2}‎\qquad‎e_{6}=Je_{3}‎
‎\end{align*}‎‎‎
Therefore ‎$‎\{e_{i}, Je_{i}\}_{i=1,\cdots,3}‎$ ‎is a‎ ‎local ‎orthogonal ‎frame ‎on‎ $M$.‎\cite{Mor2}‎
\begin{definition}\cite{Hei3}
Let $f:M^{2n}(\langle \rangle,J)\longrightarrow\mathbb{Q}^{2n+p}$ be an isometric immersion from a nearly K\"ahler manifold into a space form with second fundamental form $\alpha$ and $0 \neq\eta\in T^{\perp}_{f}M$ is a non-zero normal vector field on $M$. The umblic distribitaion of $f$ defined by $x\mapsto\Delta_{x}$
 where
\begin{eqnarray*}
\Delta_{x}=\{X\in T_{x}M|\quad\forall Y\in T_{x}M \ \alpha(X,Y)=\langle X,Y\rangle\eta\}
\end{eqnarray*}
complexification of this distribution is described by $\Delta_{x}\cap\Delta_{x}^{'}=\Delta_{x}\cap J\Delta_{x}$ where
\begin{eqnarray*}
\Delta^{'}_{x}=\{X\in T_{x}M|\quad\forall Y\in T_{x}M\ \alpha(JX,Y)+\alpha(X,JY)=0\}
\end{eqnarray*}
Now we put
\begin{eqnarray*}
\Delta^{''}_{x}=\{X\in T_{x}M|\quad\forall Y\in T_{x}M\ \alpha(T(X,Y),Z)+\alpha(X,T(Y,Z))=0\}
\end{eqnarray*}
and define by $D_{x}=\Delta_{x}\cap\Delta_{x}^{'}\cap\Delta_{x}^{''}$ the umblic distribution which is complex and invariant by the torsion of intrinsic  Hermitian connection.
\end{definition}
It is easy to see that
\begin{eqnarray*}
D_{x}=\{X\in T_{x}M|\  X, JX\in \Delta_x, \ \forall Y\in T_{x}M \ T(X,Y)\in \Delta_{x}\}
\end{eqnarray*}
 \begin{theorem}\label{main theorem 1}\cite{Hei3}
Let $f:M^{2n}\longrightarrow\mathbb{Q}^{2n+p}_{c}$ be an isometric immersion from a complete, simply connected strictly nearly K\"ahler manifold into a space form of constant curvature $c$, then there is an involute  umbilic
complex foliation on $M$ invariant by the torsion of the intrinsic Hermitian connection whose leaves are 6-nearly K\"ahler locally homogeneous manifolds (each leaf is an Amrose-Singer manifold). Moreover,
each leaf coincides with a 6-dimensional nearly K\"ahler factor appearing in the Nagy decomposition.
\end{theorem}
\begin{remark}
In proposition \ref{pro3} we described the structure of umblic distribution which is complex and invariant by the torsion of intrinsic  Hermitian connection and we got $0 \neq\eta\in T^{\perp}_{f}M$. Also we proved that in Nagy decomposition, a 6-dimensional term appears if and only if there exists $0 \neq\eta\in T^{\perp}_{f}M$ such that $D_{x}$ is a non-zero distribution. In the next section we prove more properties of $\eta$ and leaves of the foliation generated by $D$.
\end{remark}
\section{Main results}
\begin{theorem}\label{main theorem 2}
 Each leaf $N$ of the complex and invariant umbilic foliation in $M$ is minimal and the second fundamental form $\beta$ of $N$ satisfies $\beta(X,JY)=J\beta(X,Y)$. Also $\eta=3H$ where $H$ is the mean curvature of $N^{6}\hookrightarrow M^{2n}\longrightarrow\mathbb{Q}^{2n+p}_{c}$ and $c+||\eta||=\frac{S}{30}$, where $S$ is the scalar curvature of $N$ which is constant.
 \end{theorem}
 \begin{proof}
 Let $\beta$ be the second fundamental form $N^{6}\hookrightarrow M^{2n}$. The distribution $D$ is invariant under torsion of intrinsic  Hermitian connection therefore
  $$ T(X,Y)=(\nabla_{X}J)Y\in T_{x}N\quad (X,Y\in T_{x}N)$$
  and 
   $$(\nabla_{X}J)Y=\nabla_{X}JY-J\nabla_{X}Y=\nabla^{'}_{X}JY-J\nabla{'}_{X}Y+\beta(X,JY)-J\beta(X,Y)\in T_{x}N$$
 where $\nabla^{'}$ is the Levi-Civita connection on $N$. Therefore $\beta(X,JY)-J\beta(X,Y)=0$. If $\{e_{i},Je_{i}\}$ is an orthonormal local frame on $N$ then
\begin{align*}
 H=\Sigma_{i=1}^{i=3}\beta(e_{i},e_{i})+ \beta(Je_{i},Je_{i})=\Sigma_{i=1}^{i=3}\beta(e_{i},e_{i})-\beta(e_{i},e_{i})=0
 \end{align*}
where $H$ is the mean curvature of $N^{6}\hookrightarrow M^{2n}$. Therefore $N$ is minimal in $M$.\\
We show that $N$ has constant type $c+\langle \eta,\eta\rangle$. Using formulas in lemma \ref{gray formula} and Guass equation for submanifold $M$ in the space form $\mathbb{Q}$ we have
\begin{align*}‎‎
 \|T^{N}(X,Y)\|^{2}&=\|T^{M}(X,Y)\|^{2}\\
&=‎\|(\nabla_{X}J)Y\|^{2}=-\langle R_{X,Y}X,Y\rangle+\langle R_{X,Y}JX,JY\rangle\\‎
‎&=\langle \alpha(X,Y),\alpha(Y,X)\rangle-\langle \alpha(X,X),\alpha(Y,Y)\rangle‎\\&c(\langle X,Y\rangle\langle Y,X\rangle-\langle X,X\rangle\langle Y,Y\rangle\\‎&-\langle \alpha(X,JY),\alpha(Y,JX)\rangle+\langle \alpha(X,JX),\alpha(Y,JY)\\‎&c(\langle X,JY\rangle\langle Y,JX\rangle-\langle X,JX\rangle\langle Y,JY\rangle\\
‎&=(\|\eta\|^{2}+c)(-\langle X,Y\rangle^{2}+\langle X,X\rangle\langle Y,Y\rangle\\&+\langle X,JY\rangle\langle Y,JX\rangle)\\‎
‎&=(\|\eta\|^{2}+c)(\|X\|^{2}\|Y\|^{2}-\langle X,Y\rangle^{2}-\langle JX,Y\rangle^{2})‎,
‎\end{align*}‎
for all $X,Y\in\mathcal{X}(N)$. 
Hence by Proposition \ref{6-dim}, $N$ is 6-dimensional manifold and $c+\|\eta\|^{2}=\frac{S}{30}$ where $S$ is the scalar curvature of $N$. Also $\eta$ has constant length because $N$ is Einstein. By the definition of the tangent bundle $TN$ at each point we have $H=\eta$ where $H$ is the mean curvature vector field of $N$ as a submanifold of $\mathbb{Q}$.
 \end{proof}
 A complex and invariant umbilic foliation may be trivial. But in the next proposition we show that if a 6-dimensional factor in the Nagy decomposition appears then there exist ‎$‎\eta\in ‎\Gamma‎(T_{f}^{\perp}M)‎$ ‎such ‎that ‎the‎ complex and invariant umbilic foliation defined by ‎$‎\eta‎$ ‎is ‎non-trivial. Proof ‎of this ‎claim ‎needs ‎next ‎lemma‎.
\begin{lemma}‎‎\label{lemma 12}‎
Let ‎‎$(‎M,g,J)‎$ be a‎ ‎nearly ‎K\"ahler ‎manifold ‎and ‎‎$‎N‎$ be an‎ ‎almost ‎Hermitian embedded ‎submanifold ‎of ‎‎$‎M‎$ with ‎the second ‎fundamental ‎form $\beta$, ‎then for all vector fields ‎$‎X,Y$ on $N‎$‎ we have $\beta(X,JY)=J\beta(X,Y)‎$.
\end{lemma}
\begin{proof}‎‎
 Denote by ‎$‎T^{M}‎$ ‎and ‎‎$‎T^{N}‎$ the ‎torsion tensors ‎of the ‎canonical ‎Hermitian ‎connection ‎on ‎‎$M‎$ ‎and ‎‎$‎N$, respectively‎. We have
‎\begin{align}‎‎\label{relation 1}
T^{M}(X,Y)=T^{N}(X,Y)+\beta(X,JY)-J\beta(X,Y)‎‎\qquad (‎X, Y\in ‎‎\mathcal{X}(N))‎‎
‎\end{align}‎‎
‎The torsion ‎of the ‎canonical ‎Hermirion ‎connection is skew-symmetric and ‎$‎\beta‎$ is symmetric, hence
‎\begin{align*}‎‎
J\beta(X,Y)-\beta(X,JY)=\beta(JX,Y)-J\beta(X,Y)\\‎
‎\Rightarrow ‎2J\beta(X,Y)=\beta(JX,Y)+\beta(X,JY).‎
‎\end{align*}‎‎
Therefore  ‎
‎$$2J\beta(JX,JY)=-\beta(JX,Y)-\beta(X,JY)=-2J\beta(X,Y)$$
that's $‎\beta(JX,JY)=-\beta(X,Y)‎$.\\
‎Also $\beta(J(JX),JY)=-\beta(JX,Y)$, hence $‎\beta(X,JY)=\beta(JX,Y)$, and thus $‎J\beta(X,Y)=‎‎\beta(X,JY)=\beta(JX,Y)‎‎$.
‎\end{proof}‎‎
There is a short proof of the above lemma when ‎‎$‎N‎$ ‎is ‎nearly ‎K\"ahler. By ‎lemma ‎‎\ref{nijenhus} we have ‎
$$JT^{M}(X,Y)=N_J(X,Y)=JT^{N}(X,Y)‎‎$$
‎and the result follows from relation (‎\ref{relation 1}‎‎).
\begin{proposition}‎\label{pro3}
Let ‎$‎f:(M^{2n},g,J)‎\longrightarrow‎\mathbb{Q}^{2n+p}_{c}‎$ ‎be an‎ ‎isometric ‎immersion ‎from a nearly ‎K\"hler ‎manifold ‎into a space ‎form‎. If ‎$‎M‎$ ‎has an‎ ‎‎embedded 6-dimensional strictly nearly K\"aler ‎submanifold $N$  then there exists ‎$‎\eta\in ‎\Gamma‎(T^{\perp}_{f})N‎$ ‎such ‎that ‎‎$‎N‎$ ‎is ‎locally ‎isometric with a ‎leaf ‎of ‎the complex ‎and ‎invariant ‎umbilic foliation in direction ‎$‎\eta‎$. In particular‎,$N$ ‎is ‎locally ‎homogeneous.‎
\end{proposition}
\begin{proof}‎
Let $\beta‎$‎‎ be the second fundamental form of $N$ as a submanifold of $M$. By lemma ‎\ref{lemma 12} ‎for ‎all ‎‎$‎X,Y\in ‎\mathcal{X}(N)‎$ ‎we ‎have ‎‎$‎\beta(X,Y)=J\beta(X,Y)‎$‎. Now consider ‎$‎N‎$ ‎as a ‎submanifold ‎of ‎‎$‎{Q}^{2n+p}_{c}‎$ ‎(by ‎‎$‎N‎\longrightarrow ‎M‎\longrightarrow‎‎\mathbb{Q}‎$‎) and let ‎$‎H‎$ ‎be the ‎mean ‎curvature ‎of ‎‎$‎N‎$ ‎in ‎‎$‎‎\mathbb{Q}‎$. ‎We ‎choose ‎‎$‎\eta‎$ ‎parallel ‎to ‎‎$‎H=\sum_{i=1}^{3}‎\alpha‎(e_{i},e_{i})+‎\alpha‎(Je_{i},Je_{i})‎$ ‎such ‎that ‎‎$‎c+\|\eta\|^{2}=‎\frac{S}{30}‎$, ‎where ‎‎$‎\{e_{i},Je_{i}\}_{i=1, 2, 3}‎$ ‎is an ‎adapted ‎frame by ‎lemma ‎‎\ref{adapted frame} ‎for ‎‎$‎N‎$, ‎‎$‎S‎$ ‎is ‎the scalar ‎curvature ‎of ‎‎$‎N‎$ and $‎\alpha‎‎$ ‎is the second fundamental form of ‎$‎M‎$ ‎as a ‎submanifold ‎of ‎‎$‎‎\mathbb{Q}‎$ ‎via ‎‎$‎f‎$‎‎‎. Therefore ‎$‎\eta\in \Gamma‎(T^{\perp}_{f})‎$ ‎and ‎‎for ‎all ‎‎$‎X,Y\in‎\mathcal{X}(N)‎$‎, ‎‎$‎T^{M}(X,Y)=T^{N}(X,Y)‎$‎ where ‎$‎T^N‎$‎ and ‎$‎T^M‎$‎ are the torsion tensors of the canonical Hermitian connections on ‎$‎N‎$‎ and ‎$‎M‎$‎, respectively. Next, ‎‎$\nabla J‎$ ‎has ‎type ‎$‎‎\frac{S}{30}‎$ on $N$ so by a computation as in the proof of theorem \ref{main theorem 2} we have ‎‎$‎‎\alpha‎(X,Y)=\langle X, Y\rangle\eta‎$, ‎for ‎all ‎‎$‎X,Y\in ‎\mathcal{X}(N)‎$‎. Now the result follows by ‎theorem \ref{main theorem 1}‎.‎
\end{proof}‎‎
\begin{remark}‎
If there is a 6-dimensional factor $N$ in the Nagy decomposition of a nearly K\"aher manifold ‎$‎M‎$  ‎then ‎‎$‎N‎$ ‎is a‎n ‎‎‎‎embedded 6-dimensional strictly nearly K\"ahler  ‎submanifold of ‎$‎M‎$ ‎and by following ‎proposition there exist ‎$‎\eta\in ‎\Gamma‎(T^{\perp}_{f})N‎$ ‎such ‎that the ‎complex ‎and ‎invariant ‎umbilic distribution in direction ‎$‎\eta‎$ ‎is ‎non-trivial ‎and $‎N‎$ ‎is ‎locally ‎isometric ‎with a ‎leaf ‎of the foliation and so it is ‎locally ‎homogeneous.‎‎
\end{remark}‎
%the only thing remain that  we need to show the leaf of  ‎of ‎complex ‎and ‎invariant ‎umbilic foliation is strict nearly K\"ahler ( non-k\"ahler) manifolds. (see remark \ref{remark 1}).‎
\begin{proposition}‎
Let ‎$‎f:(M^{2n},g,J)‎\longrightarrow‎\mathbb{Q}^{2n+p}_{c}‎$ ‎be an‎ ‎isometric ‎immersion ‎from a complete strictly ‎nearly ‎K\"ahler ‎manifold ‎into a‎ space ‎form‎. Then there is a direction such that the ‎complex ‎and ‎invariant ‎umbilic distribution is not trivial each leaf of ‎the foliation is strictly 6-dimensional nearly K\"ahler manifold.‎‎
\end{proposition}
\begin{proof}‎
If a leaf $N$ of our foliation is ‎K\"ahler, ‎then ‎‎$‎T^{N}(X,Y)=T^{M}(X,Y)=(\nabla_{X}J)Y‎$ and  ‎‎
 \begin{align*}‎‎
‎\|(\nabla_{X}J)Y\|^{2}=&-\langle R_{X,Y}X,Y\rangle+\langle R_{X,Y}JX,JY\rangle\\‎
‎&=\langle \alpha(X,Y),\alpha(Y,X)\rangle-\langle \alpha(X,X),\alpha(Y,Y)\rangle‎\\&c(\langle X,Y\rangle\langle Y,X\rangle-\langle X,X\rangle\langle Y,Y\rangle\\‎&-\langle \alpha(X,JY),\alpha(Y,JX)\rangle+\langle \alpha(X,JX),\alpha(Y,JY)\\‎&c(\langle X,JY\rangle\langle Y,JX\rangle-\langle X,JX\rangle\langle Y,JY\rangle\\
‎&=(\|\eta\|^{2}+c)(-\langle X,Y\rangle^{2}+\langle X,X\rangle\langle Y,Y\rangle\\&+\langle X,JY\rangle\langle Y,JX\rangle)\\‎
‎&=(\|\eta\|^{2}+c)(\|X\|^{2}\|Y\|^{2}-\langle X,Y\rangle^{2}-\langle J
X,Y\rangle^{2})‎,
‎\end{align*}‎‎
where $‎‎\alpha‎‎$ is the ‎second ‎fundamental ‎form ‎‎‎of ‎‎$‎f‎$. Hence ‎$‎\|\eta\|^{2}+c‎=0‎$‎. Now ‎if ‎‎$‎c=0‎$ then the ‎space ‎form ‎is ‎isometric ‎with the ‎Euclidean ‎space and ‎‎$‎\eta=0‎$ ‎and ‎‎$‎‎\alpha‎(X,Y)=0‎$, ‎for ‎all ‎‎$‎X,Y\in \mathcal{X}(N)‎$‎. By using the above relation again, ‎we ‎have ‎‎$‎T(X,Y)=0‎$, for ‎$‎X\in T_{x}N‎$ , ‎‎‎‎$‎Y\in ‎T_{x}M‎$, which ‎means ‎that the ‎tensor ‎‎$‎r‎$ ‎on ‎‎$‎M‎$ ‎has ‎non-zero kernel and this is a contradiction, because ‎$‎M‎$ ‎is strictly nearly ‎K\"ahler. If ‎‎$‎c\neq 0‎$, we consider ‎$‎N‎$ ‎as a ‎submanifold ‎of ‎‎$‎‎\mathbb{Q}‎$ (‎by  ‎‎$‎‎N‎\hookrightarrow‎ ‎M‎\longrightarrow‎‎\mathbb{Q}‎‎$). By the Guass equation and equality ‎‎‎$‎\alpha(X,Y)=\langle X,Y\rangle\eta‎$, for ‎$‎X,Y\in \mathcal{X}N‎$,‎
\begin{align}‎‎\label{relation 2}‎‎
K(e_{i},e_{j})+K(e_{i},Je_{j})=\langle \beta(e_{i},e_{j})‎,\beta(e_{i},e_{j})‎\rangle- \langle \beta(e_{i},e_{i})‎,\beta(e_{j},e_{j})‎\rangle
\end{align}
\[‎‎‎‎‎‎‎+\langle \beta(e_{i},Je_{j})‎,\beta(e_{i},Je_{j})‎\rangle-\langle \beta(e_{i},e_{i})‎,\beta(Je_{j},Je_{j})‎\rangle=2\|\beta(e_{i},e_{j})\|^{2}‎\geq ‎0‎,
‎‎\]
where ‎$‎\beta‎$ ‎is ‎the second ‎fundamental ‎form of ‎‎$‎N‎$ ‎as a‎ submanifold ‎of ‎‎$‎M‎$‎ and ‎$‎\{e_{i},Je_{i}\}_{i=1,\cdots,3}‎$ ‎is an ‎adapted ‎frame. ‎By ‎lemma ‎‎\ref{lemma 12} ‎and similar ‎computations as above, we have ‎ ‎‎
‎$$K(e_{i},Je_{j})=K(Je_{i},e_{j}),‎\qquad ‎K(e_{i},e_{j})=K(Je_{i},Je_{j})‎\qquad ‎K(e_{i},Je_{i})=\|\beta(e_{i},e_{i})\|^{2}‎\geq ‎0‎.$$‎‎‎
Hence the scalar curvature ‎$‎S‎$ ‎of‎ ‎$‎N‎$ must be ‎non-negative,‎
$$S=\sum_{i=1}^{3}‎K‎(e_{i},e_{j})+K(e_{i},Je_{j})+K(Je_{i},Je_{j})‎\geq ‎0‎,$$‎‎‎‎‎
But by theorem 7 in ‎\cite{Fer}, ‎the scalar ‎curvature ‎of an isometrically immersed ‎K\"ahler ‎submanifold of a  ‎space ‎form ‎‎satisfies the following inequality 
$$S‎‎\leq‎ ‎2(2n)^{2}(c+\|H\|^{2})‎,$$
hence the ‎scalar ‎curvature ‎of ‎‎$‎N‎$ ‎must ‎be ‎non-positive by ‎theorem ‎\ref{main theorem 2} ‎and ‎lemma ‎‎\ref{lemma 12}‎‎. Thus ‎‎$‎S=0‎$‎. Hence ‎‎‎‎$‎\beta‎\equiv ‎0‎$‎ by (‎‎\ref{relation 2}) ‎‎and ‎‎$‎N‎$ ‎must ‎be ‎locally ‎isometric to an Euclidean ‎affine ‎space‎. This contradicts the completeness of ‎$‎M‎$.‎‎
\end{proof}‎
‎\paragraph{}
 An open connected subset $U\su M$ is saturated if each leaf of the complex and invariant umbilic foliation in $U$ is maximal in $M$. If $M$ is complete we may put $U=M$. We consider the quotient space $V=U/D$ of the leaves in $U$ (each leaf is an equivalence class) with the projection map $\pi:U^{2n}\longrightarrow V^{2n-6}=U^{2n}/D$. In general, $V$ is not a manifold. It could fail to be Hausdorff and it may be a $V$ or $QF$-manifold. But if each leaf of the complex and invariant umbilic foliation in $U^{2n}$ is complete then $V$ becomes a manifold \cite{Hei4}. In this case, $V$ is called the foliation space of the complex and invariant umbilic foliation.\\
 In \cite{Hei4} we equipped $V$ with a suitable metric and complex structure such that $V$ becomes to a quasi-K\"ahler manifold and we used this in parametrization of Euclidean nearly K\"ahler submanifolds. In the next theorem we introduce a metric and an almost complex structure on $V$ such that it becomes a nearly K\"ahler manifold.
 \begin{remark}
 In the next theorem we will use Rummler-Sullivan criterion: there exist on a compact manifold $M$ a suitable metric such that the leaves defined on $M$ are minimal if and only if there exist an $m$-form $\chi$ positive on the leaves and relativity closed such that $d\chi(X_{1},\dots,X_{m},Y)=0$ where $X_{1},\dots,X_{m}$ are tangent to the foliation \cite{Rum,Sul}.
 \end{remark}
 \begin{theorem}
 Let $f:M^{2n}\longrightarrow\mathbb{R}^{2n+p}$ be an isometric immersion from a nearly K\"ahler manifold into the Euclidean space. If each leaf of the complex and invariant umbilic foliation is complete then the leaf space is an almost complex manifold whose top cohomological group is non-trivial. If $V$ (the leaf space) is compact and $M$ is complete then by choosing a suitable metric on $V$, the projection map $\pi:M\longrightarrow V$ is a Riemannian submersion such that $\pi\circ J^{M}=J^{V}\circ M$, that is 
 $\pi$ is an almost Hermitian submersion. In particular, $V$ is a nearly K\"ahler manifold.
 \end{theorem}
 \begin{proof}
 Each leaf is minimal so if $\nu$ denotes the metric volume form of a leaf then it defines a non-trivial class in the basic cohomology $H(M/D=V)$. Indeed, let $\chi$  be the $m$-form given by Rummler-Sullivan criterion and suppose that $\nu=d\tau$ where $\tau\in \Omega^{2n-7}(V)$ then
$$\chi\wedge\nu=\chi\wedge d\tau=(-1)^{m}\{d(\chi\wedge\tau)-d\chi\wedge\tau\}$$
But $d\chi\wedge\tau=0$ because $\chi$ is relativity closed. Therefore $\chi\wedge\nu$ is exact and this is a contradiction, because $\chi\wedge\nu$ is a volume form of the complete nearly K\"ahler manifold $M$.(note that $M$ is compact and orientable).\\
Now $V$ is invariant by the almost complex structure $J$ so it will be an almost complex manifold. By the structure of the foliation space $V$ and using the projection map $\pi$ we may put a metric on $V$, inherited from metric from $M$, such that $\pi$ becomes a Riemannian submersion, preserving the almost complex structures, i.e., $\pi_*\circ J^{M}=J^{V}\circ \pi_*$. Therefore $\pi$ is an almost Hermitian submersion. Finally we show that $V$ is a nearly K\"ahler manifold. If $X,Y$ are vector fields on $M$ which are $\pi$-related to $X^{'},Y^{'}$ on $V$ then
$$\omega^{M}(X,Y)=g^{M}(X,JY)=g^{V}(X^{'},JY^{'})\circ\pi=(\pi^{*}\omega^{V})(X,Y).
$$
Therefore, on the horizontal distribution $\omega^{M}$ and $d\omega^{M}$ coincide with $\pi^{*}\omega^{V}$ and $\pi^{*}(d\omega^{V})$, respectively. Now $h((\nabla_{X}J)Y)$ is a vector field $\pi$-related to $(\nabla_{X^{'}}J)Y^{'}$ by proposition 3.5 of \cite{Fal2} where $h$ denote the horizontal component of a vector field that tangent to $M$, and we know that $\pi^{*}$ is a linear isomorphism on invariant differential forms so by proposition \ref{definition nk} we conclude that $V$ is a nearly K\"ahler manifold.
 \end{proof}
 Next we need some facts about Riemannain submersions:
  \begin{definition}
Let $(M,g,J)$ and $(B,g^{'},J^{'})$ be almost Hermitian manifolds. A Riemannain submersion $\pi:M\longrightarrow B$ is called an almost Hermitian submersion if it is an almost complex map i.e.,  $\pi_{*}\circ J=J^{'}\circ\pi_{*}$. An almost Hermitian submersion is a nearly K\"ahler submersion if the total space is nearly K\"ahler.
\end{definition}
The tangent bundle of $M$ can be decomposed as the Whitney sum of vertical distribution $V=Ker\pi_{*}$ and the complementary orthogonal metric $H$ (the horizontal distribution). Denote by $v$ and $h$ the projection on vertical and horizontal distributions, respectively. The O'Neill tensors $B$ and $A$ are defined by
\begin{align*}
&B_{X}Y=h(\nabla_{vX}vY)+v(\nabla_{vX}hY),\\
& A_{X}Y=v(\nabla_{hX}hY)+h(\nabla_{hX}vY).
\end{align*}
Here $B$ acts on each fiber (each leaf in our case) as the second fundamental form and  $A$ measures the distance of horizontal distribution from integrability.
\paragraph*{}
With the same argument as in the proposition 1 in \cite{Fal2} we can prove the following.
\begin{theorem}
Let $f:M^{2n}\longrightarrow\mathbb{R}^{2n+p}$ be an isometric immersion from a nearly K\"ahler manifold into the Euclidean space. the orthogonal complement of the complex and invariant umbilic distribution with respect to a Riemannian metric  is an integrable distribution and each integral submanifold of this distribution is totally geodesic in $M$.
\end{theorem}
\begin{proof}
 It is enough to show that the O'Neill tensor $A$ on $M$ vanishes. Let $X$ be a vector field and $V$ be a vertical field on $M$. Since
 \begin{align}\label{1}
 h(\nabla_{V}X)=h(\nabla_{X}V)=A_{X}V
 \end{align}
  we have$\nabla_{V}X=B_{V}X+A_{X}V$. Moreover, $B_{V}JX=JB_{V}X$ and we have $T(V,X)=\nabla_{V}JX-J\nabla_{V}X=A_{X}JV-JA_{X}V$. In \cite{Fal2} it was proved that for a quasi-K\"ahler submersion,
\begin{align}\label{2}
  A_{JX}=-J\circ A_{X}=A_{X}\circ J.
\end{align}
Thus
 \begin{align}\label{3}
  T(V,X)=\nabla_{V}JX-J\nabla_{V}X=A_{JX}V-JA_{X}V=-2JA_{X}V,
  \end{align}
  using (\ref{1}) and (\ref{2}) we have
  \begin{align}\label{4}
  h(T(V,X))=h(\nabla_{V}JX-h(J\nabla_{V}X)=A_{JX}V-JA_{X}V=-2JA_{X}V,
  \end{align}
  but on nearly K\"ahler manifolds the torsion of intrinsic Hermitian connection $T$ is skew-symmetric thtat is, $T(X,V)=-T(V,X)$ and by (\ref{3}) and (\ref{4}) we conclude that $JA_{X}V=0$, that is, $A=0$.
\end{proof}
\begin{corollary}\label{main Corollary}
 Let $f:M^{2n}\longrightarrow\mathbb{R}^{2n+p}$ be an isometric immersion from a complete, simply connected nearly K\"ahler manifold into the Euclidean space.  There exist two integrable distributions $D, D^{\perp}$ on manifold $M$ such that leaves of the generated foliations by $D$ and $D^{\perp}$ are minimal and totally geodesic, respectively. Moreover, if for each $X\in D_{x}^{\perp}$ and $U\in D_{x}$ we have
 \begin{equation}
 R(X,JX,U,JU)=0,
 \end{equation}
 then locally $M$ is a product of  the foliation space of complex and invariant umbilic foliation and a 6-dimensional homogeneous nearly K\"ahler manifold which is isometric with the  corresponding factor in the  Nagy decomposition.
\end{corollary}
\begin{proof}
This is proved like proposition 3.21 in \cite{Fal2}.
\end{proof}
 \begin{remark}
  Note that the 6-dimensional factor in the Nagy decomposition is not necessarily homogeneous, therefore  this is a positive answer to problem \ref{pro1} of the introduction under the assumption of Corollary \ref{main Corollary}.
 \end{remark}

\end{document}